\documentclass{amsart}

\usepackage{amsrefs}
\usepackage{amssymb}

\newtheorem{theorem}[equation]{Theorem}
\newtheorem{lemma}[equation]{Lemma}

\newtheorem{conjecture}[equation]{Conjecture}

\newcommand{\eref}[1]{(\ref{e.#1})}
\newcommand{\tref}[1]{Theorem \ref{t.#1}}
\newcommand{\lref}[1]{Lemma \ref{l.#1}}
\newcommand{\cref}[1]{Corollary \ref{c.#1}}

\numberwithin{equation}{section}
\numberwithin{figure}{section}

\newcommand{\R}{\mathbf{R}}
\newcommand{\N}{\mathbf{N}}
\newcommand{\E}{\mathbf{E}}
\newcommand{\V}{\mathbf{V}}
\renewcommand{\P}{\mathbf{P}}
\newcommand{\id}{\operatorname{1}}
\newcommand{\ep}{\varepsilon}
\newcommand{\tr}{\operatorname{tr}}
\newcommand{\mat}[1]{{\left( \begin{matrix} #1 \end{matrix} \right)}}

\begin{document}

\title{Random band matrix localization by scalar fluctuations}

\author{Nixia Chen}

\address{Department of Mathematics, The University of Chicago, Chicago, IL}

\email{nixiachen@uchicago.edu}

\author{Charles K Smart}

\address{Department of Mathematics, Yale University, New Haven, CT}

\email{charles.smart@yale.edu}

\begin{abstract}
We show the eigenvectors of a Gaussian random band matrix are localized when the band width is less than the $1/4$ power of the matrix size.  Our argument is essentially an optimized version of Schenker's proof of the $1/8$ exponent.
\end{abstract}

\maketitle

\section{Introduction}

\subsection{Some context and prior work}

A random band matrix is a random matrix whose entries vanish outside of a band around the diagonal.  The standard random band matrix model uses independent Gaussian entries: Fix integers $N \geq W \geq 1$ and sample $A \in \R^{N \times N}$ according to the probability density $a \mapsto Z^{-1} e^{-\frac{W}{4} \tr a^* a}$ on the space of matrices $a \in \R^{N \times N}$ that satisfy $a^* = a$ and $a_{i,j} = 0$ for $|i - j| > W$.  It is conjectured that, when $\ep > 0$ small and $N \geq C_\ep$ large, the eigenvectors of $A$ are typically localized when $W \leq N^{1/2-\ep}$ and typically delocalized when $W \geq N^{1/2+\ep}$.  See Casati, Molinari, and Izrailev \cite{Casati-Molinari-Izrailev}.

At the time of writing, the rigorous state of the art for this model was localization when $W \leq N^{1/7-\ep}$ and delocalization when $W \geq N^{3/4+\ep}$.  These were proved by Peled, Schenker, Shamis, and Sodin \cite{Peled-Schenker-Shamis-Sodin} and Bourgade, Yau, and Yin \cite{Bourgade-Yau-Yin}, respectively.  In this article we advance the localization regime to $W \leq N^{1/4-\ep}$ by optimizing the proof of the $1/8$ exponent in Schenker \cite{Schenker}.  In a simultaneous and independent work, Cipolloni, Peled, Schenker, and Shapiro \cite{Cipolloni-Peled-Schenker-Shapiro} also prove localization when $W \leq N^{1/4-\ep}$.

There is a related integrable band matrix model where the $N^{1/2}$ threshold is known to be sharp.  See Shcherbina \cite{Shcherbina} and Shcherbina and Shcherbina \cite{Shcherbina-Shcherbina}.  There is also a recent breakthrough on delocalization when $W \geq N^\ep$ for a $d \geq 8$ dimensional toroidal band model.  See Yang, Yau, and Yin \cites{Yang-Yau-Yin-1, Yang-Yau-Yin-2}.

\subsection{Main result}

For convenience we work with random block tridiagonal matrices.  We expect our methods can be easily adapted to handle standard random band matrices.  We sample a random block tridiagonal matrix
\begin{equation*}
A = \mat{ A_{1,1} & A_{1,2} \\ A_{1,2}^* & A_{2,2} & \ddots \\ & \ddots & \ddots & A_{N-1,N} \\ & & A_{N-1,N}^* & A_{N,N} } \in (\R^{M \times M})^{N \times N}
\end{equation*}
according to the probability density
\begin{equation}
\label{e.hamiltonian}
a \mapsto Z^{-1} e^{- \frac{M}{4} \tr a^* a}
\end{equation}
on the space of block matrices $a \in (\R^{M \times M})^{N \times N}$ that satisfy $a^* = a$ and $a_{i,j} = 0$ for $|i - j| > 1$.  We prove exponential off-diagonal decay of the resolvent:

\begin{theorem}
\label{t.main}
If $\ep > 0$, $|\lambda| < \ep^{-1}$, and $N, M \geq C_\ep$, then
\begin{equation}
\label{e.resolventdecay}
\E \| ((A-\lambda)^{-1})_{i,j} \|^{1-\ep} \leq e^{- M^{-3-\ep} |i-j|}
\end{equation}
holds for all $1 \leq i, j \leq N$.
\end{theorem}

Here and throughout the paper, we let $C > 1 > c > 0$ denote positive universal constants that may differ in each instance.  We use subscripts to denote dependence, so that $C_\ep > 1 > c_\ep > 0$ are positive constants that depend on $\ep$.

The resolvent bound \eref{resolventdecay} is known to imply localization of the eigenvectors when $N \geq M^{3+\ep}$.  See for example Aizenman, Friedrich, Hundertmark, and Schenker \cite{Aizenman-Friedrich-Hundertmark-Schenker}.  Since the flattened version of $A$ is an $NM \times NM$ matrix, this corresponds to localization when $W \leq N^{1/4-\ep}$ for the standard random band model.

By arguments of Schenker \cite{Schenker} and Bourgain \cite{Bourgain}, \tref{main} can be deduced from the following lower bound on the logarithmic fluctuations of the corner block of the resolvent.

\begin{lemma}
\label{l.fluctuations}
If $\ep > 0$, $|\lambda| < \ep^{-1}$, and $N, M \geq C_\ep$, then
\begin{equation}
\label{e.fluctuations}
\V \log  \| ((A - \lambda)^{-1})_{1,N} \| \geq M^{-3-\ep} N.
\end{equation}
\end{lemma}

In particular, the logarithmic fluctuation lower bound \eref{fluctuations} is the main contribution of our paper.  The reader interested in the reduction of \tref{main} to \lref{fluctuations}, which is arguably the most complicated part of the proof, is referred to Schenker \cite{Schenker} and Bourgain \cite{Bourgain}.

We remark that our only essential use of the Gaussian law \eref{hamiltonian} of $A$ is to compute the conditional law of $S_k$ given $D_{k-1}, \bar D_k, D_{k+1}, B_{k-1}, B_k$ in \lref{conditionallaw}.  We conjecture that this usage can be removed, and that the same matrices treated in Schenker \cite{Schenker} can also be handled using our methods.

\subsection{Scalar fluctuations}

We recall from Schenker \cite{Schenker} the reformulation of the lower bound \eref{fluctuations} in terms of the scalar fluctuations of a cocycle of a Markov chain.  Assume the hypotheses of \lref{fluctuations}.  Almost surely, Gaussian elimination yields the explicit formula
\begin{equation}
\label{e.product}
((A-\lambda)^{-1})_{1,N} = D_1^{-1} B_1 D_2^{-1} B_2 \cdots D_{N-1}^{-1} B_{N-1} D_N^{-1},
\end{equation}
where
\begin{equation}
\label{e.markov}
\begin{cases}
D_1 = A_{1,1} - \lambda \\
B_k = - A_{k,k+1} \\
D_{k+1} = A_{k+1,k+1} - \lambda - B_k^* D_k^{-1} B_k
\end{cases}
\end{equation}
for $k = 1, ..., N-1$.  We bound the logarithmic fluctuations of the product \eref{product} from below by the sum of the conditional logarithmic fluctuations of the norms of its terms.  First, we let
\begin{equation*}
S_k = \| D_k \|
\end{equation*}
and
\begin{equation*}
\bar D_k = \| D_k \|^{-1} D_k,
\end{equation*}
and then observe
\begin{equation*}
\V \log  \| ((A - \lambda)^{-1})_{1,N} \| \geq \E ( \V ( \log ( S_1 \cdots S_N ) | \bar D_1, ..., \bar D_N, B_1, ..., B_{N-1})).
\end{equation*}
Second, we use independence to distribute the variance through the logarithm.  Since $S_k$ and $(D_1, ..., D_{k-2}, D_{k+2}, ..., D_N, B_1, ..., B_{k-2}, B_{k+1}, ..., B_{N-1})$ are conditionally independent given $(D_{k-1}, \bar D_k, D_{k+1}, B_{k-1}, B_k)$, it follows that
\begin{equation}
\label{e.schenker}
\V( \log  \| ((A - \lambda)^{-1})_{1,N} \| ) \\
\geq \sum_{\substack{k \in 2 \N \\ 1 < k < N}} \E( \V( \log S_k | D_{k-1}, \bar D_k, D_{k+1}, B_{k-1}, B_k)).
\end{equation}
Here we dropped the odd terms in order to achieve conditional independence.  The lower bound \eref{fluctuations} follows from the inequality \eref{schenker} and the lower bound
\begin{equation}
\label{e.expectedscalarfluctuations}
\min_{1 < k < N} \E( \V( \log S_k | D_{k-1}, \bar D_k, D_{k+1}, B_{k-1}, B_k)) \geq M^{-3-\ep}.
\end{equation}
We prove the lower bound \eref{expectedscalarfluctuations} by explicitly computing the conditional law of $S_k$ and estimating the log concavity of its density.  See \lref{scalarfluctions} below.

\subsection{Discussion of optimality}

We expect the $1/4$ exponent is the best possible for any argument that relies on the scalar fluctuations of the terms in the product \eref{product}.  Indeed, for energies close to zero, we conjecture an upper bound on the logarithmic variance of the norms of the terms:

\begin{conjecture}
\label{j.optimal}
If $\ep > 0$ is small, $|\lambda| < \ep$, $M \geq C_\ep$, and $1 < k < N$, then
\begin{equation*}
\E( \V( \log S_k | D_{k-1}, \bar D_k, D_{k+1}, B_{k-1}, B_k)) \leq M^{-3+\ep}.
\end{equation*}
\end{conjecture}

Intuitively, the fluctuations of the norm $S_k$ should contain only an $O(M^{-2})$ fraction of the total randomness in the matrix $D_k$.  In particular, any improvement beyond the exponent $1/4$ should give finer information about the Lyapunov exponents of the cocycle \eref{product} of the Markov chain \eref{markov}.

\subsection{Acknowledgements}

The second author was partially supported by the National Science Foundation via NSF-DMS-2137909.

\section{Random full matrix bounds}

We collect some known estimates for random full matrices.  Recall that $\| A \| = \max_{\| x \| = 1} \| A x \|$ is the operator norm of $A$, $\| A \|_F = (\tr A^* A)^{1/2}$ is the Frobenius norm of $A$, and $\tr \id_I(A)$ is the number of eigenvalues of $A$ in the interval $I \subseteq \R$.

\begin{lemma}
\label{l.randomfullmatrices}
If $E \in \R^{M \times M}$ has independent $N(0,M^{-1})$ Gaussian entries, $G = (E + E^*)/\sqrt 2$, and $H \in \R^{M \times M}$ is symmetric and deterministic, then
\begin{align}
\label{e.wegner}
& \E \tr \id_I(G + H) \leq C M |I|, \\
\label{e.inversefrobenius}
& \P( \| (G + H)^{-1} \|_F \geq t ) \leq C M t^{-1}, \\
\label{e.dotproduct}
& \E (\tr G H)^2 = 2 M^{-1} \| H \|_F^2, \\
\label{e.operator}
& \P( \| E \| \geq C ) \leq e^{-c M}, \\
\label{e.conjugatedfrobenius}
\mbox{and} \quad & \P( \| E^* H E \|_F \leq c \| H \|_F) \leq e^{-c M}
\end{align}
hold for all intervals $I \subseteq \R$ and $t > 0$.
\end{lemma}

\begin{proof}
The estimates \eref{wegner} and \eref{inversefrobenius} are part of the main theorem in Aizenman, Peled, Schenker, Shamis, and Sodin \cite{Aizenman-Peled-Schenker-Shamis-Sodin}.  The estimate \eref{dotproduct} follows from $\tr G H = \sum_{ij} \sqrt{2} E_{ij} H_{ij}$ and the fact that the $E_{ij}$ are independent $N(0,M^{-1})$ Gaussians.  The estimate \eref{operator} is a standard operator norm bound for Wigner matrices and can be found in Tao \cite{Tao}.  The final estimate \eref{conjugatedfrobenius} is somewhat non-standard, so we give an ad hoc proof.  Since the law of $E$ is invariant under left multiplication by an orthogonal matrix, we may assume $H$ is diagonal and therefore $(E^* H E)_{ij} = \sum_k E_{ki} H_{kk} E_{kj}$.  For $i \neq j$, compute
\begin{equation*}
\E (E^* H E)_{ij}^2 = M^{-2} \| H \|_F^2
\end{equation*}
and
\begin{equation*}
\E (E^* H E)_{ij}^4 = 3 M^{-4} \| H \|_F^4.
\end{equation*}
It follows that
\begin{equation*}
\P ( (E^* H E)_{ij}^2 \geq c M^{-2} \| H \|_F^2 ) \geq c.
\end{equation*}
Since $(E^* H E)_{i_1 j_1}, ..., (E^* H E)_{i_m j_m}$ are independent whenever $\{ i_1, j_1 \}, ..., \{ i_m, j_m \}$ are disjoint, we can partition the entries $\{ (E^* H E)_{ij}: i \neq j \}$ into at most $C M$ sets of at least $c M$ independent entries.  In particular, the event $ \| E^* H E \|_F \leq c \| H \|_F$ is contained in a union of at most $C M$ events of probability at most $e^{-c M}$.
\end{proof}

The diagonal and off-diagonal blocks $A_{k,k}$ and $A_{k,k+1}$ have, respectively, the same law as $G$ and $E$ from \lref{randomfullmatrices}.  We can therefore use the random full matrix bounds in \lref{randomfullmatrices} to control the typical sizes of the matrices appearing in the Markov chain \eref{markov}:

\begin{lemma}
\label{l.matrixbounds}
If $\ep > 0$, $|\lambda| < \ep^{-1}$, $M \geq C_\ep$, and $1 \leq k < N$, then, with probability at least $1 - M^{-c_\ep}$,
\begin{align*}
& M^{1/2-\ep} \leq \| A_{k+1,k+1} \|_F \leq M^{1/2+\ep}, \\
& M^{1/2-\ep} \leq \| B_k^* D_k^{-1} B_k \|_F \leq M^{1 + \ep}, \\
& |\tr A_{k+1,k+1} (\lambda + B_k^* D_k^{-1} B_k)| \leq M^{1/2+\ep}, \\
\mbox{and} \quad & |\tr A_{k+1,k+1} B_k^* D_k^{-1} B_k| \leq M^{1/2+\ep}.
\end{align*}
\end{lemma}

\begin{proof}
We prove a series of estimates valid for all $\ep > 0$, $|\lambda| < \ep^{-1}$, $M \geq C_\ep$, and $1 \leq k < N$.  For brevity, we omit these quantifiers from the notation.  However, it is important to note that an estimate for some $\ep > 0$ may follow from a previous estimate for a different $\ep > 0$.  The estimate
\begin{equation*}
\P(\| A_{k+1,k+1} \|_F \geq M^{1/2+\ep}) \leq M^{-c_\ep}
\end{equation*}
follows from Markov's inequality and $\E \| A_{k+1,k+1} \|_F^2 = M+1$.  The estimate
\begin{equation*}
\P ( \tr \id_{[-M^{-\ep},M^{-\ep}]}(A_{k+1,k+1}) \geq \ep M ) \leq M^{-c_\ep}
\end{equation*}
follows from the Wegner estimate \eref{wegner}.  The estimate
\begin{equation*}
\P(\| A_{k+1,k+1} \|_F \leq M^{1/2-\ep}) \leq M^{-c_\ep}
\end{equation*}
follows from the previous estimate and the fact that $A_{k+1,k+1}$ has $M$ eigenvalues.  The estimate
\begin{equation*}
\P(\| D_k^{-1} \|_F \geq M^{1+\ep}) \leq M^{-c_\ep}
\end{equation*}
follows from the independence of $D_k - A_{k,k}$ and $A_{k,k}$ and the inverse Frobenius norm bound \eref{inversefrobenius}.  The estimate
\begin{equation*}
\P(\| B_k^* D_k^{-1} B_k \|_F \geq M^{1+\ep}) \leq M^{-c_\ep}
\end{equation*}
follows from the previous estimate and the operator norm bound \eref{operator} applied to $B_k$.  The estimate
\begin{equation*}
\P ( \tr \id_{[-M^{-\ep},M^{-\ep}]}(D_k) \geq \ep M ) \leq M^{-c_\ep}
\end{equation*}
follows from the independence of $D_k - A_{k,k}$ and $A_{k,k}$ and the Wegner estimate \eref{wegner}.  The estimate
\begin{equation*}
\P(\tr \id_{[-M^\ep,M^\ep]} (B_k^* D_k^{-1} B_k ) \leq (1 - \ep) M) \leq M^{-c_\ep}
\end{equation*}
follows from the previous estimate and the operator norm bound \eref{operator} applied to $B_k$.  The estimate
\begin{equation*}
\P(\tr \id_{[-M^\ep,M^\ep]} (D_k ) \leq (1 - \ep) M) \leq M^{-c_\ep}
\end{equation*}
follows from the previous estimate, $|\lambda| < \ep^{-1}$, and the operator norm bound \eref{operator} applied to $A_{k,k}$.  The estimate
\begin{equation*}
\P(\| D_k^{-1} \|_F \leq M^{1/2-\ep}) \leq M^{-c_\ep}
\end{equation*}
follows from the previous estimate and the fact that $\| D_k^{-1} \|_F^2$ is the sum of the squares of the eigenvalues of $D_k^{-1}$.  The estimate
\begin{equation*}
\P(\| B_k^* D_k^{-1} B_k \|_F \leq M^{1/2-\ep}) \leq M^{-c_\ep}
\end{equation*}
follows from the previous estimate, the independence of $B_k$ and $D_k$, and the conjugated Frobenius norm bound \eref{conjugatedfrobenius}.  Finally, the last two estimates in the statement of the lemma follow from the Frobenius inner product bound \eref{dotproduct}, Markov's inequality, and the independence of $A_{k+1,k+1}$ and $B_k^* D_k^{-1} B_k$.
\end{proof}

\section{Fluctuation lower bound}

For an arbitrary random vector whose law has a continuous density, we recall the conditional law of its norm given its direction:

\begin{lemma}
\label{l.normdensity}
If the random vector $X \in \R^n$ has continuous density $\phi$, $Y = \| X \|$, $\bar X = \| X \|^{-1} X$, and $f \in C_c((0,\infty))$, then
\begin{equation*}
\E( f(Y) | \bar X) = Z^{-1} \int_0^\infty  f(y) y^{n-1} \phi(y \bar X) \,dy,
\end{equation*}
where $Z = \int_0^\infty y^{n-1} \phi(y \bar X) \,dy.$ \qed
\end{lemma}

The conditional law of $S_k$ can now be computed from the density \eref{hamiltonian} of the random band matrix.

\begin{lemma}
\label{l.conditionallaw}
For $1 < k < N$ and $f \in C_c((0,\infty))$,
\begin{equation*}
\E( f(S_k) | D_{k-1}, \bar D_k, D_{k+1}, B_{k-1}, B_k) = Z_k^{-1} \int_0^\infty f(s) e^{-\phi_k(s)}  \,ds,
\end{equation*}
where
\begin{equation}
\begin{aligned}
\label{e.conditionallaw}
\phi_k(s S_k)
& = \frac{M}{4} \| s D_k + \lambda + B_{k-1}^* D_{k-1}^{-1} B_{k-1}  \|_F^2 \\
& \quad + \frac{M}{4} \| D_{k+1} + \lambda + s^{-1} B_k^* D_k^{-1} B_k \|_F^2 \\
& \quad - \frac{M^2+M-2}{2} \log (s S_k)
\end{aligned}
\end{equation}
and $Z_k = \int_0^\infty e^{-\phi_k(s)} \,ds.$
\end{lemma}

\begin{proof}
Since the law of $A$ has density \eref{hamiltonian} and the change of variables $A \mapsto (D,B)$ in \eref{markov} preserves Lebesgue measure, the law of $(D,B)$ has density
\begin{equation*}
(d,b) \mapsto Z_{N,M}^{-1} e^{- \frac{M}{4} \| d_1 \|_F^2 - \frac{M}{4} \sum_k \| d_{k+1} + \lambda + b_k^* d_k^{-1} b_k \|_F^2 - \frac{M}{2} \sum_k \| b_k \|_F^2}
\end{equation*}
on the space of pairs of sequences of matrices $(d,b) \in (\R^{M \times M})^N \times (\R^{M \times M})^{N-1}$ that satisfy $d_k^* = d_k$.  The conditional law of $D_k$ given $(D_{k-1}, D_{k+1}, B_{k-1}, B_k)$ therefore has density
\begin{equation*}
d \mapsto Z^{-1} e^{- \frac{M}{4} \| d + \lambda + B_{k-1}^* D_{k-1}^{-1} B_{k-1} \|_F^2 - \frac{M}{4} \| D_{k+1} + \lambda + B_k^* d^{-1} B_k \|_F^2}
\end{equation*}
where $Z > 0$ is $(D_{k-1}, D_{k+1}, B_{k-1}, B_k)$-measurable.  Conclude using \lref{normdensity} together with the fact that the space of $d \in \R^{M \times M}$ with $d^* = d$ has dimension $M(M+1)/2$.
\end{proof}

We estimate the typical growth of the logarithmic density $\phi_k$.

\begin{lemma}
\label{l.uniformlylogconcave}
If $\ep > 0$, $|\lambda| < \ep^{-1}$, $M \geq C_\ep$, and $1 < k < N$, then, with probability at least $1 - M^{-c_\ep}$, the logarithmic density $\phi_k$ defined in \eref{conditionallaw} satisfies
\begin{equation*}
S_k \phi_k'(S_k s) \geq M^{2-\ep} s \quad \mbox{for } s \geq M^\ep,
\end{equation*}
\begin{equation*}
S_k \phi_k'(S_k s) \leq - M^{2-\ep} s^{-3} \quad \mbox{for } 0 < s \leq M^{-\ep},
\end{equation*}
and
\begin{equation*}
|S_k^2 \phi_k''(S_k s)| \leq M^{3+\ep} (1 + s^{-4}) \quad \mbox{for } s > 0.
\end{equation*}
\end{lemma}

\begin{proof}
Using the recursion \eref{markov}, compute
\begin{equation*}
s D_k + \lambda + B_{k-1}^* D_{k-1}^{-1} B_{k-1} = s A_{k,k} + (1 - s) (\lambda + B_{k-1}^* D_{k-1}^{-1} B_{k-1})
\end{equation*}
and
\begin{equation*}
D_{k+1} + \lambda + s^{-1} B_k^* D_k^{-1} B_k = A_{k+1,k+1} + (s^{-1} - 1) B_k^* D_k^{-1} B_k.
\end{equation*}
Inserting these into the definition of $\phi_k$ in \eref{conditionallaw}, obtain
\begin{equation*}
\phi_k(S_k s) = \alpha_1 s^2 + \alpha_2 (s-1)^2 - 2  \alpha_3 s (s-1) + \alpha_4 + \alpha_5 (s^{-1}-1)^2  + 2 \alpha_6 (s^{-1}-1)  - \alpha_7 \log (S_k s),
\end{equation*}
where
\begin{align*}
& \alpha_1 = \tfrac{M}{4} \| A_{k,k} \|^2_F, \\
& \alpha_2 = \tfrac{M}{4} \| \lambda + B_{k-1}^* D_{k-1}^{-1} B_{k-1} \|_F^2, \\
& \alpha_3 = \tfrac{M}{4} \tr A_{k,k} (\lambda + B_{k-1}^* D_{k-1}^{-1} B_{k-1}), \\
& \alpha_4 = \tfrac{M}{4} \| A_{k+1,k+1} \|_F^2 \\
& \alpha_5 = \tfrac{M}{4} \| B_k^* D_k^{-1} B_k \|_F^2, \\
& \alpha_6 = \tfrac{M}{4} \tr A_{k+1,k+1} B_k^* D_k^{-1} B_k, \\
\mbox{and} \quad
& \alpha_7 = \tfrac12 (M^2 + M - 2).
\end{align*}
Using \lref{matrixbounds}, the bounds
\begin{align*}
& M^{2-\ep} \leq \alpha_1 \leq M^{2+\ep}, \\
& M^{2-\ep} \leq \alpha_2 \leq M^{3+\ep}, \\
& |\alpha_3| \leq M^{3/2+\ep}, \\
& M^{2-\ep} \leq \alpha_4 \leq M^{2+\ep}, \\
& M^{2-\ep} \leq \alpha_5 \leq M^{3+\ep}, \\
& |\alpha_6| \leq M^{3/2+\ep}, \\
\mbox{and} \quad
& M^{2-\ep} \leq \alpha_7 \leq M^{2+\ep},
\end{align*}
hold with probability at least $1 - M^{-c_\ep}$.  Without loss of generality, assume these bounds hold almost surely.  Compute
\begin{equation*}
S_k \phi_k'(S_k s) = 2 \alpha_1 s + 2 \alpha_2 (s-1) - 4 \alpha_3 s + 2 \alpha_5 (s^{-2} - s^{-3}) - 2 \alpha_6 s^{-2}  - \alpha_7 s^{-1}.
\end{equation*}
and use the bounds on $\alpha_k$ to estimate
\begin{equation*}
S_k \phi_k'(S_k s) \geq M^{2-\ep} s - M^{2+\ep} \quad \mbox{for } s > 1
\end{equation*}
and
\begin{equation*}
S_k \phi_k'(S_k s) \leq M^{2+\ep} + M^{2-\ep} (s^{-2} - s^{-3}) \quad \mbox{for } 0 < s < 1.
\end{equation*}
Conclude the first and second inequalities in the lemma statement.  Compute
\begin{equation*}
S_k^2 \phi_k''(S_k s) = 2 \alpha_1 + 2 \alpha_2 - 4 \alpha_3 + 2 \alpha_5 (3 s^{-4} - 2 s^{-3}) + 4 \alpha_6 s^{-3} + \alpha_7 s^{-2}
\end{equation*}
and use the bounds on the $\alpha_k$ and $s^{-2} + s^{-3} \leq 2 + s^{-4}$ to conclude the third inequality in the lemma statement.
\end{proof}

We prove an elementary logarithmic variance bound.

\begin{lemma}
\label{l.logvariance}
If $Y$ is a positive random variable, the law of $Y$ has continuous density $e^{-\psi}$, and there are $y_0 > 0$, $\alpha \geq 1$, and $\beta \geq 1$ such that
\begin{equation*}
y_0 \psi'(y_0 y) \geq \beta y \quad \mbox{for } y \geq \beta,
\end{equation*}
\begin{equation*}
y_0 \psi'(y_0 y) \leq - \beta y^{-3} \quad \mbox{for } 0 < y \leq \beta^{-1},
\end{equation*}
and
\begin{equation*}
|y_0^2 \psi''(y_0 y)| \leq \alpha (1 + y^{-4}) \quad \mbox{for } y > 0,
\end{equation*}
then $\V \log Y \geq c \alpha^{-1} \beta^{-6}.$
\end{lemma}

\begin{proof}
Since $\V \log (y_0^{-1} Y) = \V \log Y$ and $y_0^{-1} Y$ has density $y \mapsto y_0 e^{- \psi(y_0 y)}$, we can rescale to make $y_0 = 1$.  Using $\psi'(y) \geq \beta y$ for $y \geq \beta$, compute
\begin{align*}
&  \int_{2 \beta}^\infty |\log y| e^{-\psi(y)} \,dy \\
& \leq \int_{2 \beta}^\infty y e^{-\psi(2 \beta) + \frac12 \beta (2 \beta)^2 - \frac12 \beta y^2} \\
& = \beta^{-1} e^{-\psi(2 \beta)} \\
& \leq \beta^{-2} \int_\beta^{2 \beta} e^{-\psi(y)} \,dy \\
& \leq \beta^{-2}.
\end{align*}
Similarly, using $\psi'(y) \leq - \beta y^{-3}$ for $0 < y \leq \beta^{-1}$, compute
\begin{align*}
&  \int_0^{(2 \beta)^{-1}} |\log y| e^{-\psi(y)} \,dy \\
& \leq (2 \beta)^{-2} \int_0^{(2 \beta)^{-1}} y^{-3} e^{-\psi((2 \beta)^{-1}) + \frac12 \beta (2 \beta)^2 - \frac12 \beta y^{-2}} \\
& = (2 \beta)^{-2} \beta^{-1} e^{-\psi((2 \beta)^{-1})} \\
& \leq \frac12 \beta^{-2} \int_{(2 \beta)^{-1}}^{\beta^{-1}} e^{-\psi(y)} \,dy \\
& \leq \frac12 \beta^{-2}.
\end{align*}
Conclude $\E |\log Y| \leq \log \beta + C$ and therefore $y_1 = \exp \E \log Y$ satisfies $c \beta^{-1} \leq y_1 \leq C \beta$.  Using $|\psi''(y)| \leq \alpha (1 + y^{-4})$, compute
\begin{equation*}
\int_0^{y_1-t} e^{-\psi} + \int_{y_1+t}^\infty e^{-\psi} \geq e^{-C \alpha \beta^4 t^2} \int_{y_1-t}^{y_1+t} e^{-\psi} 
\end{equation*}
for $0 < t < c \beta^{-1}$.  Using Markov's inequality, compute
\begin{equation*}
t^{-2} \V \log Y
\geq \P(|\log Y - \log y_1| \geq t )
\geq \P(|Y - y_1| \geq C \beta t)
\geq e^{- C \alpha \beta^{6} t^2}
\end{equation*}
for $0 < t < c \beta^{-1}$.  Conclude by setting $t^{-2} = \alpha \beta^6$.
\end{proof}

Combining the previous two lemmas yields the main result.

\begin{lemma}
\label{l.scalarfluctions}
If $\ep > 0$, $|\lambda| < \ep^{-1}$, $M \geq C_\ep$, and $1 < k < N$, then, with probability at least $1 - M^{-c_\ep}$,
\begin{equation*}
\V ( \log S_k | D_{k-1}, \bar D_k, D_{k+1}, B_{k-1}, B_k ) \geq M^{-3-\ep}.
\end{equation*}
In particular, the lower bound \eref{expectedscalarfluctuations} holds.
\end{lemma}

\begin{proof}
This is immediate from \lref{uniformlylogconcave} and \lref{logvariance}.
\end{proof}

\end{document}